\documentclass[final]{siamart0516}

\usepackage[english]{babel}
\usepackage{pgf}
\usepackage{tikz}
\usepackage[utf8]{inputenc}
\usepackage{amsmath,amssymb,amsfonts}
\usepackage{enumerate}
\usepackage{blkarray}
\usepackage{cite}
\usepackage{mathtools}

\usetikzlibrary{calc,fit,matrix,arrows,automata,positioning,patterns}

\usepackage{hyperref}
\usepackage{xcolor}
\hypersetup{
    colorlinks,
    linkcolor={blue!50!black},
    citecolor={blue!50!black},
    urlcolor={blue!80!black}
}

\usepackage{pgfplots}
\usepackage{hyperref}

\newcommand\x{\times}

\newcommand\z{\colorbox{gray}{$\times$}}

\DeclareMathOperator{\rank}{rank}
\DeclareMathOperator{\rk}{rank}
\DeclareMathOperator{\diag}{diag}
\DeclareMathOperator{\blkdiag}{blkdiag}

\newcommand{\tikzmark}[1]{\tikz[overlay,remember picture] \node (#1) {};}
\newcommand{\tikzdrawbox}[3][(0pt,0pt)]{%
    \tikz[overlay,remember picture]{
    \draw[#3]
      ($(left#2)+(-0.3em,0.9em) + #1$) rectangle
      ($(right#2)+(0.2em,-0.4em) - #1$);}
}

\newcommand\TheTitle{Efficient Reduction of Compressed  Unitary plus Low-rank Matrices to Hessenberg form}
\newcommand\TheAuthors{R. Bevilacqua, G.M. Del Corso, L. Gemignani}

\headers{\TheTitle}{\TheAuthors}

\title{\TheTitle\thanks{The research of the last two authors was partially supported by GNCS project ``Tecniche innovative per
problemi di algebra lineare'' and by the project sponsored by  University of Pisa
under the grant PRA-2017-05.}}

\author{R. Bevilacqua\thanks{Dipartimento di Informatica, Universit\`a
    di Pisa, Pisa, Italy, \email{roberto.bevilacqua@unipi.it}} \and  G.M. Del Corso\thanks{Dipartimento di Informatica, Universit\`a
    di Pisa, Pisa, Italy, \email{gianna.delcorso@unipi.it}} \and   L. Gemignani\thanks{Dipartimento di Informatica, Universit\`a
    di Pisa, Pisa, Italy, \email{luca.gemignani@unipi.it}}}

\ifpdf
\hypersetup{
  pdftitle={\TheTitle},
  pdfauthor={\TheAuthors}
}
\fi

\begin{document}

\maketitle

\begin{abstract}
We present fast  numerical methods for computing the Hessenberg reduction of a unitary plus low-rank matrix
$A=G+U V^H$, where $G\in \mathbb C^{n\times n}$ is a unitary matrix represented  in some compressed format using
$O(nk)$  parameters and $U$ and $V$ are $n\times k$ matrices with $k< n$. At the core of these methods  is
a certain structured decomposition, referred to as a LFR decomposition,  of $A$ as product of three possibly perturbed  unitary $k$
Hessenberg matrices of size $n$. It is shown that in most interesting cases  an initial LFR decomposition of $A$ can be computed  very cheaply.
Then we prove structural properties of LFR  decompositions  by giving conditions under which the LFR decomposition of $A$   implies  its Hessenberg  shape.
Finally, we  describe a bulge chasing scheme for  converting  the initial LFR decomposition of $A$ into the LFR decomposition of a Hessenberg matrix
by means of unitary transformations.
The  reduction can be performed at the overall computational cost of $O(n^2 k)$ arithmetic operations using $O(nk)$  storage.
The computed LFR decomposition of the Hessenberg reduction of $A$  can be processed by the fast QR algorithm presented in~\cite{BDG}
in order to compute  the eigenvalues of $A$  within the same costs.
    \end{abstract}


\pagestyle{myheadings}
\thispagestyle{plain}
\markboth{R.~BEVILACQUA, G.~M.~DEL CORSO, L.~GEMIGNANI}{ Efficient Data-Sparse Representations of Unitary plus Low-Rank Matrices}

\section{Introduction}

Eigenvalue computations for  small rank modifications   of
unitary matrices represented in some compressed format  is a classical topic in structured numerical linear algebra.
Matrices of the form  $A = D + UV^H$ where $D$ is a unitary  $n \times n$ 
  block  diagonal matrix and 
  $U, V \in \mathbb{C}^{n \times k}$, $k< n$,  arise commonly in the numerical treatment of 
  structured (generalized) eigenvalue problems \cite{ACL,AG}. In particular any unitary plus low-rank matrix
  can be reduced in this form  by a  similarity (unitary) transformation and additionally matrices of this
  form can be directly generated by linearization techniques based on  interpolation schemes
  applied for the solution of nonlinear eigenvalue problems \cite{Austin,ito,BR,noi}.
  The class of  unitary block  upper Hessenberg matrices perturbed in the first block row or in the last block column  includes  
  block companion linearizations of matrix polynomials.  These matrices  are also related with
  computational problems involving orthogonal
  matrix polynomials on the unit circle \cite{sinap1,sinap2}. Constructing  the sequence of orthogonal polynomials   w.r.t a
  different  basis  modifies the compressed format  of the  unitary part by replacing the block Hessenberg shape with the block
  CMV shape \cite{cantero2003five,killip2007cmv,simon}. Semiinfinite block  upper Hessenberg  and CMV  unitary matrices are
  commonly used to represent unitary operators on a separable Hilbert space \cite{arli,dam}.
  Finite truncations of these matrices  are
  unitary block  Hessenberg/CMV matrices  modified in the  last row  or column. 
  
  In most numerical methods Hessenberg  reduction by unitary similarity transformations is the first step  towards eigenvalue
  computation. Recently  a  fast reduction algorithm specifically tailored for block companion matrices
  has been presented in \cite{AMRVW_block} whereas some  efficient  algorithms for dealing with
  block unitary diagonal plus small rank matrices  
  have been developed  in \cite{GR}.  In particular,   these latter algorithms are two-phase:
  in the first phase the matrix $A$ is reduced in
  a banded form $A_1$  employing  a block CMV-like format to represent the unitary part.  The second phase amounts 
  to incrementally annihilate the  lower subdiagonals of $A_1$  by means of Givens rotations  which are gathered in order
  to construct  a data-sparse  compressed representation of
  the  final Hessenberg matrix $A_2$.  The representation involves $O(nk)$  data storage  consisting of
  $O(n)$ vectors of length
  $k$ and $O(nk)$  Givens rotations.  This compression  is usually known as 
   a Givens--Vector representation \cite{vanbarel:book1,vanbarel:book2}, 
   and  it can  also be   explicitly  resolved to   produce a
   generators-based representation \cite{eidelman:book1,eidelman:book2}.  However,
   a major weakness of this  approach is that   both  these
    two compressed formats   are  not suited to be exploited  in the design of 
  fast  specialized eigensolvers for unitary plus low rank matrices using $O(n^2 k)$ ops only. 

  In this paper  we  describe a novel $O(n^2 k)$ backward stable algorithm for computing the Hessenberg reduction of general  matrices
  $A\in \mathbb C^{n\times n}$  of the form $A = G + UV^H$, where  $G$ is  unitary block diagonal or
 unitary block upper Hessenberg or block CMV   with block size $k<n$ and,  in the case $G$ is  unitary block upper Hessenberg or block CMV, we have the additional requirement that  $U=\left[I_k,0\ldots, 0\right]^T$. These families include most of the important cases arising in applications.

   This  algorithm  circumvents  the  drawback   of the method proposed in \cite{GR}
  by  introducing  a different  data-sparse  compressed representation of
  the  final Hessenberg matrix which is effectively  usable in fast eigenvalue   schemes. In particular,
  the representation  is   suited   for the fast
eigensolver for unitary plus low rank matrices  developed in \cite{BDG}. Our derivation is
  based on three  key ingredients or building blocks:
  \begin{enumerate}
  \item A condensed representation of the matrix $A$ (or of a matrix unitary similar to $A$) which can be specified as $A=L( I + (e_1 \otimes I_k)Z^H)R=LFR$, 
where $L$ is the product of $k$ unitary lower Hessenberg matrices, $R$
is the product of
$k$  unitary upper Hessenberg matrices and the middle factor $F$  is  the identity  matrix perturbed
in the first $k$ rows. 

In the case matrix $G$ is block upper Hessenberg or block diagonal we can obtain the $LFR$ representation in a simple way that we clarify in Section~\ref{Hess} and~\ref{Diag}. 
In the case $G$ is unitary block CMV matrix we provide  
a   suitable   extension of the well known factorization of CMV matrices as product of two block diagonal
  unitary matrices that are   both the  direct sum of $2\times 2$ or $1\times 1$ unitary  blocks (compare  with 
  \cite{simon}   and the references given therein).    Specifically, block CMV matrices with blocks of size $k$  are $2k$-banded
  unitary matrices  allowing a 'staircase-shaped'  profile. It is shown that a block CMV matrix with blocks of size $k$
  admits a factorization  as product of two unitary  block  diagonal matrices with $k\times k$   diagonal  blocks.
  It follows that the  block CMV matrix can be  decomposed as the product of a unitary lower $k-$Hessenberg matrix
  multiplied by a unitary upper $k-$Hessenberg matrix. 
\item  An embedding  technique which  for a given triple $(L,F, R)$ associated with $A$ 
  makes it possible to construct  a larger matrix $\widehat A\in \mathbb C^{(n+k)\times (n+k)}$
  which is  still unitary plus rank$-k$ and  it  can be  factored  as
  $\widehat A=\widehat L\cdot \widehat F\cdot \widehat R$, where $\widehat L$ is the product of $k$
  unitary lower Hessenberg matrices, $\widehat R$
  is the product of $k$  unitary upper Hessenberg matrices and the middle factor $\widehat F$
  is   unitary  block diagonal plus rank$-k$ with   some additional properties.
\item  A theoretical result which provides  conditions under which a matrix  specified  in the form
  $\widehat A=\widehat L\cdot \widehat F\cdot \widehat R$  turns out to be Hessenberg. 
  \end{enumerate}
  Combining  together  these ingredients   allows the
  design  of  a specific bulge-chasing  strategy for converting the $ L F R$  factored representation of
  $\widehat A$  into the $ L  F R$   decomposition of an  upper Hessenberg matrix
  $\widetilde A$ unitarily similar to $\widehat A$.   
The  final  representation  of $\widetilde A$  thus involves $O(nk)$  data storage  consisting of
$O(k)$ vectors of length
$n$ and $O(nk)$  Givens rotations. The reduction to Hessenberg form turns out to have the same asymptotic complexity of eigensolvers for unitary plus low rank matrices and furthermore, this representation  is  suited to be used directly by   the fast
eigensolver for unitary plus low rank matrices  developed in \cite{BDG}.

The paper is organized as follows. In Section \ref{one} we introduce the  $LFR$ representations of unitary
plus rank$-k$ matrices by   devising  fast algorithms for transforming a  matrix $A$ into its $LFR$ format provided that $A$
belongs to some special classes.
In Section \ref{two} we investigate  the properties of $LFR$ representations of unitary plus rank$-k$
Hessenberg matrices and we describe a suitable technique to embed the matrix $A$ into a larger matrix $\widehat A$
by mantaining its structural properties. In Section \ref{three} we present our algorithm  which modifies 
the $LFR$ representation of $\widehat A$  by computing  the corresponding 
$LFR$ representation of a unitarily similar Hessenberg matrix. Finally,  numerical experiments are discussed in
Section \ref{four} whereas  conclusions  and future work are drawn in Section \ref{sec:conclusions}.

\section{The $LFR$ Format  of Unitary plus Rank-$k$ Matrices}\label{one}
In this section we introduce a suitable compressed representation of unitary plus rank-$k$ matrices which can be exploited
for the design of fast Hessenberg reduction algorithms.

\begin{definition}\label{lfr}
		A unitary plus rank-$k$ matrix $A\in \mathbb C^{n\times n}$ can be represented in the LFR format if there is a triple $(L,F,R)$ of matrices  such that:
		\begin{enumerate}
			\item $A=LFR$;
			\item $L\in \mathbb C^{n\times n}$ is the product of $k$ unitary lower Hessenberg matrices;
			\item $R\in \mathbb C^{n\times n}$ is the product of $k$ unitary upper  Hessenberg matrices;
			\item $F=Q+[I_k, 0]^T Z^H\in \mathbb C^{n\times n}$  is  a unitary plus rank$-k$ matrix, where  $Q$  is a block diagonal
			unitary matrix of the form
			$Q=\left[\begin{array}{c|c}
			I_k & \\ \hline & \hat Q
			\end{array}\right] $, with $\hat Q$  unitary
			Hessenberg and  $Z\in \mathbb C^{n\times k}$.
			
		\end{enumerate}
	\end{definition}

In the sequel of this section  we present some fast algorithms  for computing the $LFR$ format of  a unitary  plus
rank-$k$ matrix $A\in \mathbb C^{n\times n}$ specified as follows:
\begin{itemize}
\item $A=G + [I_k, 0]^T Z^H$,   $Z\in \mathbb C^{n\times k}$, and
  $G$ is unitary  block CMV   with block size $k<n$;
  \item $A=H + [I_k, 0]^T Z^H$,   $Z\in \mathbb C^{n\times k}$, and
    $H$ is unitary  block  upper Hessenberg  with block size $k<n$;
  \item $A=D + U V^H$,   $U, V\in \mathbb C^{n\times k}$, and
    $D$ is unitary  block  diagonal  with block size $k<n$.
\end{itemize}
These three cases cover the most interesting structures of low-rank perturbation of unitary matrices. In the general case of unitary matrices, where it is not known the spectral factorization of the unitary part or the unitary matrix cannot be represented in terms of a linear number of parameters,  we cannot expect to recover the eigenvalues -- even only of the unitary part -- in $o(n^3)$. 


In the following sections we investigates into the above three cases. 
\subsection{Small Rank Modifications of Unitary Block CMV Matrices}\label{CMV}
A block analogue of the CMV form of unitary matrices has been introduced in  \cite{GR,arli}.
\begin{definition}[CMV shape]\label{cmvshape}
  A unitary matrix $G\in \mathbb C^{n\times n}$ is said to be \emph{CMV structured with block size $k$}
  if there exist $k \times k$  non-singular 
  matrices $R_i$ and $L_i$, respectively upper and lower triangular,
  such that 
  \begin{equation} \label{gshape}
    G  = 
    \begin{bmatrix}
          \times & \times & L_3  \\
          R_1 & \times & \times \\
            & \times & \times & \times & L_5\\
            & R_2 & \times & \times & \times \\
            &     &        & \times & \times \\
            &     &        & R_4    & \times & \ddots \\
            &     &        &        & \ddots & \ddots \\
        \end{bmatrix}
        \end{equation}
        or
        \[
     G = 
    \begin{bmatrix}
    	\times & L_2    &  \\
    	\times & \times & \times & L_4    &  \\
    	R_1    & \times & \times & \times \\
    	       &        & \times & \times & \times &L_6    \\
    	       &        & R_3    & \times & \times & \times & \ddots\\
    	       &        &        &        & &\ddots & \ddots
    \end{bmatrix}    
  \]
  where the symbol $\times$ has been used to identify 
  (possibly) nonzero $k\times k$ blocks. 
\end{definition}

Block CMV matrices are associated with matrix  orthogonal polynomials on the unit circle and
the structure of the matrix depends on the
choice of the starting basis of the set of matrix polynomials to be orthogonalized.  In particular,  
$G$ fits  the block structure shown in Definition~\ref{cmvshape}  if $\left \{I_k, zI_k, z^{-1} I_k,\ldots\right\}$ or
$\left \{I_k, z^{-1}I_k, z I_k,\ldots\right\}$ are considered.  In what follows  for the sake of simplicity we always assume that
$G$  satisfies  the  block structure~\eqref{gshape}.  Furthermore,
in order to simplify the notation we often assume that
$n$ is a multiple
of $2k$, so the above structures fit ``exactly'' in the matrix. However,
this is not restrictive  and the theory presented here continues
to hold in greater generality. In practice, one can deal with the
more general case by allowing the blocks in the bottom-right corner
of the matrix to be smaller.

Notice that a matrix in CMV form with blocks of size $k$ is, 
in particular, $2k$-banded. The CMV structure with blocks
of size $1$ has been proposed as a generalization of what the
tridiagonal structure is for Hermitian matrices in 
\cite{cantero2003five} and \cite{killip2007cmv}. A further analogy between the
scalar and the block case is 
 derived from the Nullity Theorem \cite{FM} that is here  applied to unitary matrices. 

 \begin{lemma}[Nullity Theorem]\label{l1}
Let $U$ be a unitary matrix of size $n$. Then 	
$$
\rank(U(\alpha, \beta))=\rank(U(J\backslash \alpha,J\backslash \beta))+|\alpha|+|\beta|-n
$$
where $J=\{1, 2, \ldots, n\}$ and $\alpha$ and $\beta$ are subsets of $J$.
If $\alpha=\{1, \ldots, h\}$ an $\beta=J\backslash \alpha$ we have
$$
\rank(U(1:h, h+1:n))=\rank(U(h+1:n, 1:h)),\quad  \mbox{ for all } h=1, \ldots, n-1.
$$
\end{lemma}

 From Lemma \ref{l1} applied  to   a block CMV structured matrix $G$ of block size $k$   we find that for 
 $p>0$:
\[
0=\rk\left(G(1:2pk, (2p+1)k+1: n)\right)=\rk\left(G(2pk+1:n, 1:(2p+1)k)\right)-k
\]
which gives 
\[
\rk\left(G(2pk+1:2(p+1)k, (2p-1)k+1:(2p+1)k)\right)=k.
\]
Pictorially we are observing  rank  constraints  on the following blocks
\[
 G = \begin{bmatrix}
          \times & \times & L_3  \\
          R_1 & \times & \times \\
            & \tikzmark{left1}\z & \z & \times & L_5\\
            & R_2 & \z  \tikzmark{right1}& \times & \times \\
            &     &        & \tikzmark{left2}\z & \z \\
            &     &        & R_4    & \z  \tikzmark{right2} & \ddots \\
            &     &        &        & \ddots & \ddots \\
    \end{bmatrix}
    \tikzdrawbox{2}{thick,green}
    \tikzdrawbox{1}{thick,green}
  \]
and by  similar arguments on   the corresponding blocks in the  upper triangular portion. 

In the scalar case   with $k=1$ these  conditions make it possible to find
a factorization of the  CMV matrix as product of two  block diagonal matrices
usually referred to as the classical Schur parametrization \cite{BGE}.  
Similarly, here we  introduce a block  counterpart of  the Schur 
parametrization which gives a useful tool to  encompass the structural properties of  block CMV  
representations. 

\begin{lemma}[CMV factorization]\label{lem:cmv-fact}
  Let $G$ be a unitary CMV structured matrix  with blocks
  of size $k$ as defined in Definition \ref{cmvshape}.  Then $G$  can be factored
  in two block diagonal unitary matrices
  $G = G_1 G_2$ of the form:
  \[
    G_1 = \diag(G_{1,1}, \ldots, G_{1,s}), \qquad 
    G_2 = \diag(I_k, G_{2,2}, \ldots, G_{2,s+1})
  \]
  such that
   $G_{2,s+1}$ has $k$ rows 
  and columns and all the other blocks 
  $G_{i,j}$ have $2k$ rows and columns and 
  bandwidth $k$ with  both $G_{i,j}(k+1:2k, 1:k)$ and $G_{i,j}(1:k, k+1:2k)$
  triangular  matrices of full rank. Moreover, 
  each matrix  $G$ admitting such  a factored  form is in turn CMV. 
\end{lemma}

\begin{proof}
  The proof of this result is constructive, and can be obtained by
  performing  a block QR decomposition. We notice that if we 
  compute a QR decomposition of the top-left $2k \times k$ block of 
  $G$ we have
  \[
    \begin{bmatrix}
      Q_{1,1} & Q_{1,2} \\
      R_{2,1} & Q_{2,2} \\ 
      && I
    \end{bmatrix}^H \begin{bmatrix}
              \times & \times & L_3  \\
              R_1 & \times & \times \\
                & \times & \times & \times & L_5\\
                & R_2 & \times & \times & \times \\
                &     &        & \times & \times \\
                &     &        & R_4    & \times  & \ddots \\
                &     &        &      & \ddots & \ddots
            \end{bmatrix} = \begin{bmatrix}
	            \tilde \times &   \\
	              & \tilde \times & \tilde \times \\
	              & \times & \times & \times & L_5\\
	              & R_2 & \times & \times & \times \\
	              &     &        & \times & \times  \\
	              &     &       &  R_4    & \times  & \ddots\\
                      &     &        &      & \ddots & \ddots
            \end{bmatrix}
  \]
  where $\tilde \times$ identifies the blocks that have been altered
  by the transformation and  the block in position $(1,1)$ can be assumed to be the identity matrix.
  Notice that in the first row the
  blocks in the second and third columns have to be zero due to $G$ being unitary, and that the 
$R_{2,1}$ block is nonsingular  upper triangular since it inherits the properties of 
   $R_1$. 
  
  We can continue this process by 
  computing the QR factorization of $\left[ \begin{smallmatrix}
    \times \\
    R_2 
  \end{smallmatrix} \right]$. 
Notice that,
from  the application   of the Nullity Theorem  \ref{l1} the block identified 
  by $\left[ \begin{smallmatrix}
      \times & \times \\
      R_2    & \times \\ 
    \end{smallmatrix} \right]$ 
in the picture has rank at most $k$. This also holds for all the other
    blocks for the same kind. In particular, computing the QR factorization of the
    first $k$ columns and left-multiplying by $Q^H$ will put to 
    zero also the block on the right of $R_2$. We will then get
    the following factorization: 

{\scriptsize{
    \[
    \begin{bmatrix}
      Q_{1,1} & Q_{1,2} \\
      R_{2,1} & Q_{2,2} \\ 
      && Q_{3,3} & Q_{3,4} \\
      && R_{4,3} & Q_{4,4} \\
      &&&& I \\
    \end{bmatrix}^H \begin{bmatrix}
              \times & \times &L_3  \\
              R_1 & \times & \times \\
                & \times & \times & \times & L_5\\
                & R_2 & \times & \times & \times \\
                &     &        & \times & \times \\
                &     &        & R_4    & \times & \ddots\\
                &     &        &      & \ddots & \ddots

            \end{bmatrix} = 
\begin{bmatrix}
	            \tilde \times &   \\
	              & \tilde \times & \tilde \times \\
	              & \tilde \times & \tilde \times & & \\
	              & &  & \tilde \times & \tilde\times \\
	              &     &        & \times & \times \\
	              &     &        & R_4    & \times & \ddots\\
                      &     &        &      & \ddots & \ddots

            \end{bmatrix} 
    \]
}}where we notice that, as before, the block $R_{4,3}$ is  nonsingular upper triangular 
    and that some blocks in the upper part have been set to zero thanks
    to the unitary property. The process can then be continued until
    the end of the matrix, providing  a factorization of $G$ as product of two unitary block diagonal matrices, that is
    $G= \widehat G_1 \widehat G_2$.  This factorization can  further be simplified  by means of a  block diagonal scaling
    $G=( \widehat G_1 D) (D^H \widehat G_2)=G_1 G_2$ with  $D=\diag(D_1, \ldots, D_{2s})$, $D_{2j-1}=I_k$ and
    $D_{2j}$ $k\times k$ unitary matrices determined so that  the blocks $G_{i,j}$ are  of  bandwidth $k$, that is the outermost blocks in $G_1$ and $G_2$ are triangular.  For the sake
    of illustration consider $j=1$ and  let  $Q_{1,2}^H=Q R$  be a QR decomposition of $Q_{1,2}^H$.  By setting
    $D_2=Q$ we obtain that $Q_{1,2}D_2=R^H$ and, moreover, from $L_3=Q_{1,2}D_2 (G_2)_{2,3}=R^H(G_2)_{2,3}$  it follows that
    the  block of $G_2$ in position $(2,3)$ also exhibits a lower triangular structure.  The  construction of the
    remaining  blocks  $D_{2j}$, $j>1$, proceeds in a similar way. 
\end{proof}

Pictorially, the above result gives the following 
structure of $G_1$ and $G_2$: 
\[
  G_1 = \begin{bmatrix}
    \begin{tikzpicture}
      \foreach \x in {1, ..., 3} {
        \draw (\x,4-\x) -- (\x+.5,4-\x) -- (\x+1,3.5-\x) --(\x+1,3-\x)
          -- (\x+0.5,3-\x) -- (\x,3.5-\x) -- (\x,4-\x);
      }
    \end{tikzpicture}
  \end{bmatrix}, \qquad 
  G_2 = \begin{bmatrix}
    \begin{tikzpicture}
      \draw (0,4) -- (.5,4) -- (.5,3.5) -- (0,3.5)
         -- (0,4);
      \foreach \x in {.5,1.5} {
       \draw (\x,4-\x) -- (\x+0.5,4-\x) -- (\x+1,3.5-\x) --(\x+1,3-\x)
                -- (\x+.5,3-\x) -- (\x,3.5-\x) -- (\x,4-\x);
      }
      \draw (2.5,1.5) -- (3,1.5) -- (3,1) -- (2.5,1) -- (2.5,1.5);
    \end{tikzpicture}
  \end{bmatrix}
\]

Now, let us assume that a matrix $A\in \mathbb C^{n\times n}$  is  such that $A=G^T + 
[I_k, 0]^T Z^H$,   $Z\in \mathbb C^{n\times k}$,  and
$G$ is unitary  block CMV   with block size $k<n$.  By replacing $G$ with its block diagonal
factorization we obtain that $A=G_2^T(I_n+[I_k, 0]^T Z^H \bar G_1)G_1^T$.
Since the left-hand and the right-hand side matrices  are unitary $k-$banded it follows that  they can both  be factored
as the product of $k$ unitary Hessenberg matrices.   Hence, we have the following. 

\begin{theorem}\label{block-cmv-red}
Let $A\in \mathbb C^{n\times n}$  be   such that $A=G^T + 
[I_k, 0]^T Z^H$,   $Z\in \mathbb C^{n\times k}$, and
$G$ is unitary  block CMV   with the block structure shown in Equation~\ref{gshape}.  Then
$A$ can be represented in the LFR format as $A=G_2^T(I_n+[I_k, 0]^T \widehat Z^H)G_1^T$  where
$L=G_2^T$, $R=G_1^T$, $G=G_1 G_2$ is the decomposition  provided  in  Lemma \ref{lem:cmv-fact}
and $F=I_n+[I_k, 0]^T \tilde Z^H$, $ \tilde Z^H=Z^H \bar G_1$.
\end{theorem}

The overall  cost of  computing this condensed  LFR representation  of  the  unitary plus rank-$k$ matrix $A$ is
$\mathcal O(nk^2)$ flops using
$\mathcal O(n k)$ memory storage.

\subsection{Small Rank Modifications of Unitary Block Hessenberg Matrices}\label{Hess}
The class of perturbed unitary block Hessenberg matrices includes the  celebrated block companion forms which  are the
basic tool in the construction of matrix linearizations of matrix polynomials. To be specific let $A\in \mathbb C^{n\times n}$ be
a matrix such that $A= H + [I_k, 0]^T Z^H$,   $Z\in \mathbb C^{n\times k}$, and
$H$ is unitary  block  upper Hessenberg with block size $k<n$.  A compressed LFR format of a matrix unitarily similar to $A$
can be computed as follows.  First of all we can  suppose that  all the subdiagonal blocks $H_{i+1,i}$, $1\leq i\leq n/k$,
are  upper
triangular. If not  we consider the   unitary block diagonal matrix $P$  defined by
$P=\blkdiag\left[P_1, P_2, \ldots, P_{n/k}\right]$ where $P_i\in \mathbb C^{k\times k}$, $P_1=I_k$
and $H_{i+1,i}P_i=P_{i+1}R_i$ is  a QR decomposition of the matrix  $H_{i+1,i}P_i$,
$1\leq i\leq n/k-1$.  Then the matrix  $\widetilde A=P^H A P $  is  such that $\widetilde A=\widetilde H+ [I_k, 0]^T \widetilde Z^H$
and $H$ is unitary  block  upper Hessenberg with block size $k<n$ and $\widetilde H_{i+1,i}=R_i$, $1\leq i\leq n/k-1$.  Hence,
the matrix $\widetilde H$ is banded with lower bandwidth  $k$ and therefore  the factorization
$\widetilde A=I_n( I_n+ [I_k, 0]^T \widehat Z^H)\widetilde H$ gives a suitable LFR representations of  $\widetilde A$.
Summing up we have  the following. 

\begin{theorem}\label{block-hess-red}
Let $A\in \mathbb C^{n\times n}$  be   such that $A=H + 
[I_k, 0]^T Z^H$,   $Z\in \mathbb C^{n\times k}$, and
$H$ is unitary  block  upper Hessenberg    with  block size $k<n$.  Then there exists a unitary block diagonal matrix
$P=\blkdiag\left[P_1, P_2, \ldots, P_{n/k}\right]$,  $P_i\in \mathbb C^{k\times k}$, $P_1=I_k$ such that 
$ \widetilde A=P^H A P $ can be represented in the LFR format as
$\widetilde A=I_n( I_n+ [I_k, 0]^T \widehat Z^H)\widetilde H$  where
$L=I_n$, $R=\widetilde G=P^H G P $ and $F=I_n+[I_k, 0]^T \widehat Z^H$, $ \widehat Z^H=Z^H \widetilde  H^H$.
\end{theorem}

The overall  cost of  computing this condensed  LFR representation  of  the  unitary plus rank-$k$ matrix $A$ is
$\mathcal O(nk^2)$ flops using
$\mathcal O(n k)$ memory storage.

\subsection{Small Rank Modifications of Unitary Block  Diagonal  Matrices}\label{Diag}

The  unitary block diagonal matrix reduces to a unitary diagonal matrix up to a similarity transformation which can be performed within $O(n k^2)$ operations. 
The interest toward the properties of
block CMV matrices is renewed in \cite{GR} where a  general scheme  is proposed
to transform a  unitary diagonal plus a rank$-k$
matrix into a block  CMV structured matrix plus a rank$-k$ perturbation located
in the first $k$ rows only. More specifically we have the following \cite{GR}. 

\begin{theorem}\label{thm:diagonal-to-block-cmv}
  Let $D\in \mathbb C^{n\times n}$  be  a unitary diagonal 
  matrix and $U \in \mathbb C^{n\times k}$  of full rank $k$. Then,
  there exists a unitary matrix $P$ such that
  $G=PDP^H$ is CMV  structured with block size $k$ and the block structure shown in Definition \ref{cmvshape}  and
  $PU=(e_1 \otimes I_k) U_1$ for some $U_1\in \mathbb{C}^{k\times k}$. The matrices $P, G$ and $U_1$ can be computed with $O(n^2k)$ operations. 
\end{theorem}

By applying  Theorem \ref{thm:diagonal-to-block-cmv} to the matrix pair $(D^H, U)$ we find that there exists
a unitary matrix $P$ such that
  $G=PD^HP^H$ is CMV  structured with block size $k$  and
$PU=(e_1 \otimes I_k) U_1$.  In view of  Lemma \ref{lem:cmv-fact}   this yields
\[
  \begin{array}{ll}
    P( D+UV^H) P^H= G^H +(e_1 \otimes I_k) U_1 (PV)^H =\\
   G_2^H( I + (e_1 \otimes I_k)Z^H)G_1^H,
  \end{array}
\]
where $Z=G_1^HPVU_1^H\in \mathbb{C}^{n\times k}.$
Since the left-hand and the right-hand side matrices  are unitary $k-$banded it follows that  they can both  be factored
as the product of $k$ unitary Hessenberg matrices. In this way we  obtain the next result. 

\begin{theorem}\label{block-diag-red}
Let $A\in \mathbb C^{n\times n}$  be   such that $A= D+UV^H$ with $U,V\in \mathbb C^{n\times k}$, and
$D$  unitary diagonal.  Then there exists a unitary matrix
$P\in \mathbb C^{n\times n}$ such that
$G=PDP^H$ has the  block CMV structure shown in Definition \ref{cmvshape}  and
  $PU=(e_1 \otimes I_k) U_1$ for some $U_1\in \mathbb{C}^{k\times k}$. Moreover, 
$ \widetilde A=P A P^H $ can be represented in the LFR format as
$\widetilde A=G_2^H( I_n+ [I_k, 0]^T  Z^H)\widetilde G_1^H$  where
$L=G_2^H$, $R=G_1^H $, $PDP^H=G=G_1 G_2$ is the  factorization of $G$
provided in Lemma \ref{lem:cmv-fact} and $F=I_n+[I_k, 0]^T Z^H$, $Z^H= U_1 (PV)^H$.
\end{theorem}

The overall  cost of  computing this condensed  LFR representation  of  the  unitary plus rank-$k$ matrix $A$ is
$\mathcal O(n^2k)$ flops using
$\mathcal O(n k)$ memory storage.

In the next sections  we investigate the properties of the
Hessenberg reduction of a matrix
given in the $LFR$ format.

\section{Factored Representations of  Hessenberg Matrices}\label{two}

In this section we investigate  suitable conditions under which a factored representation $A=LFR \in \mathbb C^{m\times m}$,  where
$L$ is the product of $k<n$ unitary lower Hessenberg matrices, $R$
is the product of
$k$  unitary upper Hessenberg matrices and the middle factor $F$  is  unitary plus rank$-k$,
specifies a matrix in Hessenberg form. In Section~\ref{three} we will discuss the chasing algorithm for   reducing, by unitary similarity, a matrix of the form $L(I+(e_1\otimes I_k)Z^H)R$ to Hessenberg form maintaining the factorization and enforcing the properness of the factor $L$ to avoid breakdown of the subsequent $QR$ iterations.

A key ingredient  is the properness of the generalized Hessenberg factors.
\begin{definition}
	A matrix $H\in \mathbb C^{m\times m}$ is called {\em $k$-upper Hessenberg} if $h_{ij}=0$ when $i>j+k$.
	Similarly, $H$ is called  {\em $k$-lower Hessenberg} if $h_{ij}=0$ when $j>i+k$.   In addition, 	
	when  $H$ is {\em $k$-upper Hessenberg} ({\em $k$-lower Hessenberg}) and
        the outermost entries are non-zero, that is, $h_{j+k,j}\neq 0$ ($h_{j,j+k}\neq 0$), $1\leq j\leq m-k$,
         then  the matrix is called {\em proper}.	
\end{definition}

Note that for $k=1$   a  Hessenberg matrix $H$ is proper iff it is unreduced. 
Also,  a $k$-upper Hessenberg matrix $H\in \mathbb C^{m\times m}$ is proper iff $\det(H(k+1:m, 1:m-k))\neq 0$.
Similarly a $k$-lower Hessenberg matrix $H$ is proper iff $\det(H(1:m-k, k+1:m))\neq 0$.

An important property of any unitary upper Hessenberg
matrix $H\in \mathbb C^{m\times  m} $  is that it can be represented as product of  elementary transformations,
i.e., $H=\mathcal G_1\mathcal G_2\cdots \mathcal G_{m-1} \mathcal D_m$
where $\mathcal G_\ell=I_{\ell-1}\oplus G_\ell \oplus I_{m-\ell-1}$  with $ G_\ell=\left[\begin{array}{cc} \alpha_\ell & \beta_\ell \\
- \beta_\ell & \bar \alpha_\ell\end{array}\right]$, $|\alpha_\ell|^2 + \beta_\ell^2=1$, $\alpha_\ell, \in \mathbb C$, $\beta_\ell\in \mathbb{R}, \beta_\ell\ge 0$,
are unitary Givens rotations and $\mathcal D_m=I_{m-1}\oplus \theta_m$ with $|\theta_m|=1$. In this way  the matrix $H$ is stored by
two vectors of length $m$  formed by the elements $\alpha_\ell, \beta_\ell$, $1\leq \ell\leq m-1$ and $\theta_m$.
The same representation also extends  to  unitary $k$-upper   Hessenberg  matrices specified as the product of $k$
unitary upper Hessenberg matrices multiplied on the right by a unitary diagonal matrix which is the identity matrix modified in the last diagonal entry. Lower unitary Hessenberg matrices can be parametrized similarly as $H=\mathcal G_{m-1}\mathcal G_{m-2}\cdots \mathcal G_{1} \mathcal D_m$.

Another basic property of unitary plus rank$-k$ matrices is the existence  of suitable embeddings which maintain their structural properties.  The embedding turns out to be crucial to ensure the properness of the factor $L$ and guarantee the safe application of implicit $QR$ iterations. The embedding is also important for the bulge chasing algorithm as we explain in the next section. The following result is  first proved in  \cite{BDG} and here specialized  to a matrix of the form determined in  Theorems  \ref{block-cmv-red},  \ref{block-hess-red} and  \ref{block-diag-red}.

\begin{theorem}\label{theo:embedding}
  Let $A\in \mathbb{C}^{n\times n}$ be such that $A=L( I + (e_1 \otimes I_k)Z^H)R=LFR$,
  where $L$ and $R$ are unitary  and $Z\in \mathbb C^{n\times k}$.  Let
  $Z=Q G$, $G\in \mathbb C^{k\times k}$, be  the economic QR factorization of $Z$.  Let 
  $\widehat U\in \mathbb C^{m\times m}$, $m=n+k$,  be defined  as
  \[
  \widehat U=I_m -\left[\begin{array}{cc} Q\\-I_k\end{array}\right]\left[\begin{array}{cc} Q\\-I_k\end{array}\right]^H.
  \]
  Then it holds
  \begin{enumerate}
  \item  $\widehat U$ is unitary;
  \item  the matrix $\widehat A \in \mathbb C^{m\times m}$ given by
    \[
    \widehat A=\left[\begin{array}{cc}L \\ & I_k\end{array}\right] \left( \widehat U  +
    \left(\left[\begin{array}{cc}G^H\\0\end{array}\right] + \left[\begin{array}{cc} Q\\-I_k\end{array}\right]\right)
    \left[\begin{array}{cc}Q\\0\end{array}\right]^H\right) \left[\begin{array}{cc}R \\ & I_k\end{array}\right]
    \]
    satisfies
    \[
    \widehat A=\left[\begin{array}{cc} A & B \\0 & 0\end{array}\right], \quad B\in \mathbb C^{n\times k}.
    \]
   \end{enumerate} 
\end{theorem}

\begin{proof}
  Property 1 follows  by direct calculations from
  \[
  \left[\begin{array}{cc}Q\\-I_k\end{array}\right]^H\left[\begin{array}{cc} Q\\-I_k\end{array}\right]=2I_k.
  \]
  For Property 2 we  find  that
  \[
  \widehat U  +
    \left(\left[\begin{array}{cc}G^H\\0\end{array}\right] + \left[\begin{array}{cc} Q\\-I_k\end{array}\right]\right)
    \left[\begin{array}{cc}Q\\0\end{array}\right]^H=
    \left[\begin{array}{cc}I_n & Q \\ & 0_k\end{array}\right] +
    \left[\begin{array}{cc} I_k\\0\end{array}\right]\left[\begin{array}{cc} Z\\0\end{array}\right]^H .
    \]
\end{proof}

The  unitary matrices $L$ and $R$   given  in  Theorems  \ref{block-cmv-red},  \ref{block-hess-red} and  \ref{block-diag-red} are $k$-Hessenberg matrices.  The same clearly holds for the
larger matrices $\diag(L, I_k)$ and $\diag(R, I_k)$ occurring in the factorization of $\widehat A$.  The next result
is the main contribution of this section and it provides conditions under which a matrix  specified as a product $L\cdot \tilde F\cdot R$,  where
$L$ is  a unitary $k$-lower Hessenberg matrix $R$ is a unitary $k$-upper Hessenberg matrix and $\tilde F$ is  a unitary  matrix
plus  a rank$-k$  correction,   is in Hessenberg form. 

In fact, once we apply the embedding described by Theorem~\ref{theo:embedding} to $A=L(I+(e_1\otimes I_k)Z^H)R$, the matrix obtained,  $\widehat A$, is no more in the LFR format since the middle factor is not in the prescribed format required by Definition~\ref{lfr}. Moreover $\widehat L=L\oplus I_k$ is not a proper matrix, making implicit $QR$ iterations subject to breakdown.  

\begin{theorem}\label{teo:rep1}
  Let $L, R\in \mathbb{C}^{m\times m}$, $m=n+k$,
  be two unitary matrices, where $L$ is a proper unitary
  $k$-lower Hessenberg matrix and $R$ is a unitary $k$-upper Hessenberg matrix.
  Let $Q$ be a block diagonal unitary upper Hessenberg matrix of the form
  $Q=\left[\begin{array}{c|c}
I_k & \\ \hline & \hat Q
    \end{array}\right] $, with $\hat Q$ $n\times n$ unitary
  Hessenberg. Let $F=Q+[I_k, 0]^TZ^H$  be   a unitary plus rank$-k$ matrix  with $ Z\in \mathbb C^{m\times k}$.
  Suppose that the matrix $\widehat A=LFR$ satisfies the block structure
  \[
  \widehat A=\left[\begin{array}{cc} A & \ast \\ 0_{k, n} & 0_{k,k}\end{array}\right].
  \]
  Then  $\widehat A$ is an upper Hessenberg matrix.
  \end{theorem}
\begin{proof}
  From Lemma \ref{l1}   we find that $M=L(n+1:m, 1:k)$ is nonsingular due to the properness of $L$.

  Now, let us consider the   matrix $C=L\,Q$. This matrix
 is unitary with a $k$-quasiseparable structure below the $k$-th upper diagonal. Indeed,
 for any $h, h=2, \ldots n+1$ we have
\begin{eqnarray*}
C(h:m, 1:h+k-2)&=&L(h:m, :)\, Q(:, 1:h+k-2)=\\
&=&L(h:m, 1:h+k-1)\,Q(1:h+k-1,1:h+k-2).
\end{eqnarray*}
Applying Lemma~\ref{l1} we have $\rank(L(h:m, 1:h+k-1))=k$, implying that also $\rank(C(h:m, 1:h+k-2))\le k$. 
Since $C(n+1:m, 1:k)=L(n+1:m. : ) Q(:, 1:k)=M$ is non  singular, we conclude that $\rank(C(h:m, 1:h+k-2))= k$, $2\leq h\leq n+1$.

From this observation we can then find  a set of generators $P, S\in \mathbb{C}^{(m\times k)}$ and
a $(1-k)$-upper Hessenberg matrix $U_k$
 such that $U_k(1,k)=U_k(n, m)=0$
so that $C=PS^H+U_k$ \cite{DB}.

 Then we can recover the rank $k$
correction $PS^H$ from the left-lower  corner of $C$  obtaining
\[
PS^H=C(:, 1:k) M^{-1} C(n+1:m, :)=L(:,1:k) M^{-1} C(n+1:m, :), 
\]
since $C(:, 1:k)=L Q(:, 1:k)=L(:,1:k)$.  Notice that  $B=U_k\, R$ is upper Hessenberg as it is  the product of a $(1-k)$-upper
Hessenberg matrix by  a $k$-upper Hessenberg matrix. Moreover, we find that
$B(n+1:m, :)=U_k(n+1:m, :)R=0$ since $U_k(n+1:m, :)=0$.  From the block structure of $\widehat A$ there follows that
\[
\left(C(n+1:m, :)+M  Z^H\right)R=0,
\]
which gives
\[
PS^H=L(:,1:k) M^{-1} C(n+1:m, :)=-L(:,1:k) Z^H=-L[I_k, 0]^TZ^H.
\]
Hence $U_k=L(Q+[I_k,0]^TZ^H)=L\, F$ and therefore $B=U_k\, R=LFR =\widehat A$ which concludes the proof.
\end{proof}

\section{ The Bulge Chasing  Algorithm}\label{three}
In this section we present a bulge-chasing algorithm relying upon Theorem \ref{teo:rep1}
to compute the Hessenberg reduction of the matrix
$\widehat A$ given as  in Theorem \ref{theo:embedding}, i.e.,  the embedding of  $A=L(I+(e_1\otimes I_k)Z^H)R$. We recall that $Q$ and $G$ are the  factors of the economic $QR$ factorization of $Z$. 

Let us set
\[
X=\left[\begin{array}{cc} Q\\-I_k\end{array}\right], \quad Y=
\left[\begin{array}{cc}G^H\\0\end{array}\right] + X,  \quad \quad W=\left[\begin{array}{cc} Q\\0\end{array}\right], 
\]
so that we have
\begin{equation} \label{widehatA}
\widehat A=\left[\begin{array}{cc}L \\ & I_k\end{array}\right] \left( \widehat U  +
Y W^H\right) \left[\begin{array}{cc}R \\ & I_k\end{array}\right],
\quad  \widehat U =I_m-XX^H.
\end{equation}
Observe that $X(k+1:m, :)=Y(k+1:m, :)$ and, moreover, $Y(n+1:m, :)=-I_k$  which implies $\rank(Y)=k$.
In the preprocessing phase we initialize
\[
L_0:=\left[\begin{array}{cc}L \\ & I_k\end{array}\right], \quad  R_0:=\left[\begin{array}{cc}R \\ & I_k\end{array}\right],  \quad 
X_0:=X,  \quad Y_0:=Y, \quad W_0:=W.
\]
Notice that $L_0$ is a unitary $k$-lower  Hessenberg matrix and $R_0$ is  a unitary $k$-upper  Hessenberg matrix and, therefore,
they can both  be represented by the product of $k$ Hessenberg matrices.  This property will be maintained under the
bulge chasing process. In the cases considered in this paper, we rely on the additional structure of $L_0$ namely that $L_0$ is also $k$-upper Hessenberg as we can observe from Theorems~\ref{block-cmv-red},~\ref{block-hess-red} and~\ref{block-diag-red}.	

In this section we make use of the following technical result.
\begin{lemma}\label{lem:fact}
	Let $B\in \mathbb{C}^{n\times n}$ be a  unitary $k$ Hessenberg matrix.
	Let $H\in \mathbb{C}^{n\times n}$, be a unitary Hessenberg obtained as a sequence of ascending or descending  Givens transformations acting on two consecutive rows, i.e. $H={\mathcal G}_{n-1}{\mathcal G}_{n-2}, \cdots {\mathcal G}_1$ if $H$ is lower Hessenberg  or $H={\mathcal G}_1{\mathcal G}_2 \cdots {\mathcal G}_{n-1}$ if $H$ is upper Hessenberg. 
	Then, there exist a unitary  $k$ Hessenberg matrix  $\tilde B$ (with the same orientation as $B$) and a unitary Hessenberg  matrix $\tilde H$ such that $HB= \tilde B\tilde H$ where
	\begin{itemize} \item $\tilde H=\begin{bmatrix}
	I_k & \\ &\hat H
	\end{bmatrix}$ if  $B$ is $k$-lower Hessenberg, 
	\item $\tilde H=\begin{bmatrix}
\hat H& \\ &	I_k 
	\end{bmatrix}$ if $B$ is $k$-upper Hessenberg, 
\end{itemize}
and $\hat H$ has the same orientation of $H$.
\end{lemma}
\begin{proof}
We prove the Lemma only in the case $H$ is lower Hessenberg and $B$ is $k$-upper Hessenberg. 
	We need to move each of the $n-1$ Givens rotations of $H$  on the right of $B$. The first $k$ Givens rotations of $H$, namely ${\mathcal G}_{1}, \ldots, {\mathcal G}_k$, when applied to  $B$ do not destroy the $k$-lower Hessenberg structure of $B$, so that ${\mathcal G}_k {\mathcal G}_{k-1}\cdots {\mathcal G}_1B=\hat B$ still $k$-lower Hessenberg. 
	When we apply ${\mathcal G}_{k+1}$ to $\hat B$ a bulge is produced in position $(k+2, 1)$, and we need to apply a rotation on the first two columns of ${\mathcal G}_{k+1}\hat B$ to remove the bulge, i.e. ${\mathcal G}_{k+1}\hat B=\hat B_1 \tilde {\mathcal G}_1$, similarly we can remove each of the remaining $n-k-1$ Givens rotations. At step $i$ we have
	${\mathcal G}_{k+i}\hat B_{i-1}=\hat B_i \tilde {\mathcal G}_i$. The last Givens ${\mathcal G}_{n-1}$ produces a bulge in position $(n,  n-k-1)$ which can be removed by the rotation $\tilde  {\mathcal G}_{n-k-1}$ acting on the columns $(n-k-1, n-k)$. We do not need to rotate the columns with indices between $n-k$ and $n$, so that $$\tilde H=\tilde {\mathcal G}_{n-k-1}\cdots \tilde {\mathcal G}_2 \tilde {\mathcal G}_1=\begin{bmatrix}
	\hat H& \\ &	I_k 
	\end{bmatrix}.
	$$ 
	We can similarly prove the remaining three cases.
\end{proof}

The reduction of $ \widehat A= \widehat A_0$ in  Hessenberg form proceeds in three steps according to Theorem \ref{teo:rep1}.
The first two steps amount to determine a different representation of the same matrix  $\widehat A_0$.
In particular after these two steps the rank-correction inside the brackets is confined to the first $k$-rows, while the $L_0$  factor on the left of the representation is substituted by a factor which is proper, and  still with the lower $k$-Hessenberg structure.
 The third step is a
bulge-chasing  scheme to complete the  Hessenberg reduction. 

\begin{enumerate}
\item(\emph{QR decomposition of $Y_0$}) We compute the  full QR  factorization of $Y_0=Q_0T_0$.  Since $Y_0$ is full rank the matrix $\hat T_0=T_0(1:k, :)$ is invertible   and, moreover, the matrix $Q_0$ can be taken as a $k$-lower Hessenberg proper matrix (see Lemma~2.4 of~\cite{BDG}).  We  can write
  \[
  \widehat A_0=(L_0 Q_0)\cdot ( Q_0^H \widehat U+ T_0 W_0^H)\cdot R_0.
  \]
  Then the matrix $\widehat A_1\colon =L_0^H \widehat A_0 L_0$ is such that
  \[
  \widehat A_1=Q_0 \cdot  ( Q_0^H \widehat U R_0+ T_0 W_0^HR_0) L_0.
  \]
  Notice that  $ \widehat U_1:=Q_0^H \widehat U R_0$ is  a unitary $2k$-upper   Hessenberg matrix.
  Indeed,  we have that 
  $$
  \widehat U_1=Q_0^H\widehat UQ_0 Q_0^HR_0=(I_m-\hat X\hat X^H) Q_0^HR_0,
  $$
  where $\hat X:=Q_0^H X$ and  $\hat X(2k+1:m,:)=-Q_0^H(2k+1:m, 1:k)G^H=0$ since $Q_0^H(2k+1:m, 1:k)=0$.
  Therefore, it holds  $\widehat U_1=((I_{2k}-\hat X(1:2k,:)\hat X^H(:,1:2k))\oplus I_{m-2k})Q_0^HR_0$ which, for the block diagonal structure of $I_m-\hat X\hat X^H$, turns out to be $2k$-upper Hessenberg.
\item(\emph{Block decomposition of  $\widehat U_1$}) We compute the full  QR  factorization of $\widehat U_1^H(:, 1:k)$.
  Specifically we  determine a unitary  matrix $P$ such that $\widehat U_1( 1:k, :) P=\left[I_k, 0\right]$, and such $P$ can be taken in $k$-lower Hessenberg form (see Lemma~2.4 of~\cite{BDG}).
  The matrix
  \begin{equation}
  \label{u1p}
  \widehat U_1 P=\left[\begin{array}{c|c}
  I_k & \\ \hline & U_1(k+1:m, :)P(:, k+1:m)
  \end{array}\right]=\left[\begin{array}{c|c}
I_k & \\ \hline & \hat Q
    \end{array}\right] 
\end{equation}
    where $\hat Q$ is   a unitary  $k$-upper   Hessenberg matrix, due to the fact that  $U_1(k+1:m, :)$ is $k$-upper Hessenberg and $P(:, k+1:m)$ is lower triangular.
  We obtain that
  \[
  \widehat A_1=Q_0 \cdot  ( \widehat U_1P+ T_0 W_0^HR_0P)P^H L_0,
  \]
  which gives
  \[
  \widehat A_2= L_0\widehat A_1 L_0^H=\widehat A_0=(L_0 Q_0)\cdot  ( \widehat U_1P+ T_0 W_0^H R_0P)P^H.
\]
Applying $k$ times Lemma~\ref{lem:fact}, observing that $L_0$ is $k-$banded (i.e. simultaneously $k$-upper and $k$-lower Hessenberg)  we can factorize $L_0Q_0=Q_1 L_1$ where
  $Q_1$ is  a unitary  $k$-lower    Hessenberg matrix and  $L_1=\left[\begin{array}{c|c}
I_k & \\ \hline & \hat L_1
   \end{array}\right]$  where $\hat L_1$ is a unitary  $k$-upper     Hessenberg matrix. 
    It follows that
\begin{equation} \label{A1}
  \widehat A_0= Q_1 \cdot (L_1\widehat U_1P+ T_0 W_0^HR_0P)P^H=Q_1(\widehat U_2+ (e_1\otimes I_k)W_1^H )P^H.
\end{equation}
Where the matrix $\widehat U_2:=L_1\widehat U_1P$ satisfies $\widehat U_2=\left[\begin{array}{c|c}
I_k & \\ \hline & \tilde U_2
    \end{array}\right]$ where $\tilde U_2$ is a unitary  $2k$-upper Hessenberg matrix, and $W_1:=P^HR_0^HW_0\hat T_0^H$, where $\hat T_0=T(1:k, 1:k)$. Observe that
  $Q_0(n+1:m,1:k)=Q_1(n+1:m,1:k)$  and, moreover $Q_0(n+1:m,1:k)$ is nonsingular, because $Q_0$ is proper. 
  From Lemma \ref{l1}  this implies the properness of $Q_1$.   This  property is maintained  in the subsequent
  steps of the reduction
  process so that the final  matrix is guaranteed to be proper as  prescribed  in Theorem \ref{teo:rep1}.
  
  At the end of this step the enlarged matrix $\widehat A$  has been reduced to a product of a proper $k$-lower Hessenberg matrix $Q_1$ on the left, a unitary factor corrected in the first $k$ rows i.e., the term inside the brackets, and a $k$-upper Hessenberg matrix, i.e., $P^H$. Step 3 consists of the reduction of $\hat U_2$ to Hessenberg form so that the final matrix will be unitarly  similar to $\widehat{A}$ and in the $LFR$ format.
\item(\emph{Hessenberg reduction of  $\hat U_2$})  We now need to work on the representation of   $\widehat{A}_0$ in equation~\eqref{A1} to reduce the inner matrix $\widehat U_2$ in Hessenberg form  by means of a bulge-chasing  procedure. Indeed Theorem~\ref{teo:rep1} ensures that the matrix obtained will be in the LFR format and in Hessenberg form. These transformations will not affect the properness of the $k$-lower Hessenberg term on the left.

For the sake of illustration let us
consider the first  step.  Let us determine a unitary  upper Hessenberg matrix $\mathcal G_1 \in \mathbb C^{2k\times 2k}$ such that
\[
\mathcal G_1 \tilde U_2(2:2k+1, 1)=\alpha e_1.
\]
Then setting  $G_1=(I_{k+1}\oplus {\mathcal G}_1\oplus I_{n-2k-1})$,   we have  
$$
\widehat A_0=Q_1 G_1^H (G_1\widehat U_2+(e_1\otimes I_k)W_1^H )P^H.
$$
  
The application of $G_1^H$ on the right of the matrix $Q_1$  by computing  $Q_1(:, k+2:3k+1)\mathcal G_1^H$  creates a bulge  formed by an additional segment above the last  nonzero  superdiagonal of $Q_1$.  This segment can be annihilated by a novel unitary  upper Hessenberg matrix  $ G_2$ whose active part $\mathcal{G}_2  \in \mathbb C^{2k\times 2k}$ works on  the left of
  $Q_1(:, k+2:3k+1)\mathcal G_1^H$  by acting on the rows  of indices 2 through  $2k+1$.  
We can then apply a similarity transformation to remove the bulge
 \[
 G_2\widehat A_0G_2^H= Q_2 (G_1\widehat U_2+(e_1\otimes I_k)W_1^H )P^HG_2^H,
 \]
 where $Q_2:=G_2Q_1 G_1^H $.
The active part of $G_2^H$, the $2k\times 2k$ matrix $\mathcal G_2^H$, acts  on the right of $P^H$   producing a bulge  which can  be zeroed  by a unitary  upper Hessenberg matrix $\mathcal G_3 \in \mathbb C^{2k\times 2k}$ working on rows from $k+2$ to $3k+1$ of $P^HG_2^H$. 
Then,  the matrix
  \[
  \widehat U_2\leftarrow G_1   \widehat U_2 (I_{k+1}\oplus \mathcal G_3^H\oplus I_{n-2k-1})
  \]
  has a bulge  on the rows  of indices $2k+2$ through  $4k+1$
  which can be chased away by a sequence of $O(n/k)$ transformations having the same structure as above.
  Note that the rank correction of the unitary matrix inside the brackets is never affected by these transformations so that, at the end of the process, we have unitarily reduced $A_0$ to the LFR format in Definition~\ref{lfr}. Also the zeros in the last $k$ rows are preserved.  
\end{enumerate}

The cost analysis is rather standard for matrix algorithms based on chasing operations \cite{R_book}.
\begin{enumerate}
\item Step 1 requires to compute the economic QR decomposition of a matrix of size $(n+k)\times k$  and to multiply a unitary
  $k-$Hessenberg matrix  specified as product of $k$ unitary Hessenberg matrices
  by $k$ vectors of size $n+k$.  The total cost is
  $O(nk^2)$ ops.
\item  The cost of  Step 2 is asymptotically the same.  The construction of the  factored representation of $\hat Q$ as well as
  the computation of  $L_1$ and $Q_1$  can still  be performed  using $O(nk^2)$ ops.
\item   The dominant cost  is the  execution of   Step 3. The zeroing of the sub-subdiagonal  entries costs
  $O(n \frac{n}{k}k^2)=O(n^2k)$ ops.
\end{enumerate}
In the next section we provide algorithmic details and discuss  the results of numerical experiments  confirming the effectiveness and the robustness of our proposed approach.

\section{Numerical Results}\label{four}
The structured Hessenberg reduction scheme described in the previous section
has been implemented using  MATLAB for numerical testing.
The resulting algorithm basically  amounts to manipulate chains of unitary
 Hessenberg matrices.

At step 1 of the structured Hessenberg reduction scheme   we  first compute the full QR factorization of the matrix
$Y_0\in \mathbb C^{m\times k}$.  The matrix $Q_0^H$ turns out to be  the product of $k$ unitary upper Hessenberg matrices.
Then we  have to incorporate the   unitary matrix  $\mathcal S:=I_{2k}-\hat X(1:2k,:)\hat X^H(:,1:2k)$ on the
right into the factored representations of $Q_0^H$ and $R_0$.  The unitary $2k\times 2k$ matrix $\mathcal S$ can always be represented as the product of at most $k(2k-1)$ elementary unitary transformations of size $2\times 2$.
Once this factorization is computed, we have to add each of these single transformations,  one by one, on  the right to the factored representations of $Q_0^H$ and $R_0$. This is accomplished by a sequence of  turnover and fusion operations acting on the chains of elementary transformations in $Q_0^H$ and $R_0$ (see~\cite{Va11} for the detailed description of these operations on elementary transformations).

At the beginning of step 2 the matrix $\widehat U_1$ is a $2k$-upper Hessenberg matrix, and is essentially  determined by the product of two  unitary $k$-upper  Hessenberg
matrices that here we rename as  $\widehat U_1=\widehat P \widehat Q$.
To reshape this factorization in the desired form in equation~\eqref{u1p} we can apply $k$ times a  reasoning similar to the one done in Lemma~\ref{lem:fact} to move each elementary transformation  of  $\widehat Q$ on the left. 
 In this way we find $\widehat P \widehat Q=\widetilde Q \widetilde P$ where
$\widetilde Q=\left[\begin{array}{c|c}I_k & \\ \hline & \hat Q \end{array}\right]$ is the matrix appearing in~\eqref{u1p}.
Since  $\widehat Q$ is formed by $O(nk)$ elementary transformations  the reshaping costs  $O(nk^2)$  ops. With a similar reasoning we can
compute the representations of $Q_1$ and $L_1$ where $Q_1$ is $k$-lower Hessenberg and $L_1=\begin{bmatrix}
I_k & \\
& \hat L_1
\end{bmatrix}$, with $\hat L_1$ unitary $k$-upper Hessenberg.

The third phase  of the structured Hessenberg reduction scheme basically amounts to  reduce  the matrix $\widehat U_2=L_1\widetilde Q$ into
a matrix of the form $ \left[\begin{array}{c|c}
I_k & \\ \hline & \tilde U_2
    \end{array}\right] $, with $\tilde U_2$ $n\times n$ unitary
Hessenberg. To be specific assume that $L_1=L_{1,1} \cdots L_{1,k}$  and
$\widetilde Q=\widetilde Q_1 \cdots \widetilde Q_k$, where $L_{1,j}$ and $\widetilde Q_j$ are unitary upper Hessenberg matrices
with the leading principal submatrix of order $k$ equal to the identity matrix.    The  overall reduction process
splits into $n$ intermediate steps. At each step the first active  elementary transformations of
$\widetilde Q_{k}, \ldots,\widetilde Q_1,  L_{1,k},\ldots, L_{1,1}$ are annihilated (in this order). Each transformation is moved on the
left by creating a bulge in the leftmost factor $Q_1$.  This bulge is removed by applying a similarity transformation.

Let us consider the first step. Let $L_{1,i}={\mathcal G}_{k+1}^{(i)}\cdots {\mathcal G}_{m-1}^{(i)}D_m^{(i)}$ denote the Schur parametrization of $L_{1,i}$ and similarly let $\widetilde Q_i={\mathcal H}_{k+1}^{(i)}\cdots {\mathcal H}_{m-1}^{(i)}E_m^{(i)}$ that of $\widetilde Q_i$. At this step we move left the first elementary transformations of each factor of the product $L_1\widetilde Q$, for example when moving the rotation ${\mathcal H}_{k+1}^{(k)}$ in front of $L_1$ the resulting transformation acts on rows $3k$ and $3k+1$ while some of the  rotations in $L_1$ and $\tilde Q$ have changed. The final situation is as follows\footnote{As observed, we can use only a unitary diagonal matrix to keep track of all the diagonal contributions.}
\begin{align*}
L_1\widetilde Q&=({\mathcal G}_{k+1}^{(1)}\cdots \hat {\mathcal G}_{m-1}^{(1)})\cdots ({\mathcal G}_{k+1}^{(k)}\cdots \hat {\mathcal G}_{m-1}^{(k)})
( {\mathcal H}_{k+1}^{(1)}\cdots \hat  {\mathcal H}_{m-1}^{(1)})\cdots ( {\mathcal H}_{k+1}^{(k)}\cdots \hat  {\mathcal H}_{m-1}^{(k)})D=\\
&=\underbrace{(\tilde {\mathcal H}_{3k}^{(k)}\tilde {\mathcal H}_{3k-1}^{(k-1)}\cdots \tilde {\mathcal H}_{2k+1}^{(1)} \tilde {{\mathcal G}}_{2k}^{(k)}\cdots \tilde{\mathcal G}_{k+2}^{(2k)} )}_{B} \underbrace{\tilde {\mathcal G}_{k+1}^{(1)}\tilde{L}_{1, 1}\cdots \tilde{L}_{1,k} \hat{Q}_{1}\cdots \hat{Q}_{k} D}_{\hat U_2},
\end{align*} 
where
$$
\tilde{L}_{1, j}=\tilde {\mathcal G}_{k+2}^{(j)}\cdots \tilde {\mathcal G}_{m-1}^{(j)} \quad \mbox{ and } \quad \hat{Q}_{j}=\tilde  {\mathcal H}_{k+2}^{(j)}\cdots \tilde  {\mathcal H}_{m-1}^{(j)}.
$$
At this point we bring the bulge $B$ on the left of $Q_1$ in equation~\eqref{A1} obtaining 
$$\widehat A_0= \widehat B \breve Q_1( \widehat U_2 + T_0W_0^HR_0P)P^H,
$$
where $\widehat B=\Gamma_{2k}\cdots \Gamma_2$ is the product of a sequence of elementary transformations in ascending order acting on rows $2:2k$. The bulge $\widehat B$ is removed 
by chasing an elementary transformation at a time. 
For example to remove $\Gamma_{2k}$ we apply the similarity transformation
$\Gamma_{2k}^H  \widehat B \breve Q_1(  + T_0W_0^HR_0P)P^H \, \Gamma_{2k}$  that will shift down the bulge of $2k$ positions. So $O(n/k)$ chasing step will be necessary to get rid of that first transformation. 
%
In this way the
overall process is completed using $O(nk \cdot k\cdot n/k)=O(n^2k)$ ops. Note that the whole  similarity transformation acts only on the first $n$ rows leaving untouched  the null rows at the bottom of $\widehat A$ in equation~\eqref{widehatA}.

Numerical experiments have been performed to confirm the computational prop\-er\-ties of the proposed method. Among the
three cases considered in Section \ref{one} the  last one, when the unitary part is block diagonal,  is the most challenging since
computing the  starting LFR format costs $O(n^2k)$ vs the $O(nk^2)$ flops sufficient for the first two cases. 
The CMV reduction of the input unitary diagonal plus rank$-k$ matrix  $D+UV^H$
is  computed  using the algorithm  presented in  \cite{GR}  which is fast and
backward stable.
 Our tests focus on the numerical performance of the  Hessenberg reduction scheme provided
 in the previous section given the  factors $L,R$ and $Z$  satisfying  Theorem  \ref{block-diag-red}.
 In the next tables we show the backward
errors  $\epsilon_P$, $\epsilon_B$ and $\epsilon_H$ generated by our procedure.  These errors are defined as follows:
\begin{enumerate}
\item $\epsilon_P$ is  the error computed at the end of the first two {\em preparatory} steps.  Given the matrix
  $A$  of size  $n$ represented  as in Theorem \ref{theo:embedding} we  find  the matrix $\widehat A$ of size $m=n+k$
  obtained at the end of
  step 2.  Denoting by $fl(\widehat A)$ the computed  matrix, the error is
  \[
  \epsilon_P:=\frac{\| A-fl(\widehat A(1\colon n, 1\colon n))\|_2}{mk\| A\|_2}.
  \]
\item  $\epsilon_B$ is the classical {\em backward } error  generated in the  final step given by
  \[
  \epsilon_B:=\frac{\| H- Qfl(\widehat A) Q^H\|_2}{mk\| A\|_2}, 
  \]
  where $H$ is the matrix computed by multiplying all the factors obtained at the end of the third step,   and
  $Q$ is the product of the  unitary transformations acting by similarity  on the left and on the right  of the matrix
  $fl(\widehat A)$ in the Hessenberg reduction phase.
\item  $\epsilon_H$ is  used to  measure the Hessenberg structure of the matrix $H$. It is
  \[
  \epsilon_H:=\frac{\| {\tt tril}(H,-2)\|_2}{mk\| A\|_2},
  \]
  where ${\tt tril}(X,K)$ is the  matrix  formed by the elements on and below the $K$-th diagonal of $X$.
   \end{enumerate}
Next tables report these errors  for different values of $n,k$ and $\| A\|_2$.

\begin{table} 
  \begin{center}
  \begin{tabular}{ccccc}
    n & $\parallel A\parallel_2$ & $\epsilon_P$ &  $\epsilon_B$ &  $\epsilon_H$  \\\hline\hline
    32 & 8.2e+01& 2.2e-17& 3.9e-17&4.3e-19\\
    64 & 1.4e+02& 1.5e-17& 5.2e-17& 5.1e-19\\
    128 & 2.7e+02& 7.7e-18& 6.0e-17& 2.0e-19\\
    256 &5.2e+02& 5.5e-18& 1.3e-16& 1.4e-19 \\
    512 &1.0e+03& 3.2e-18& 2.2e-16& 1.4e-19 \\ 
  \end{tabular}
  \caption{Backward errors for random matrices with $k=2$}
  \label{T12} 
  \end{center}
\end{table}

\begin{table} 
  \begin{center}
  \begin{tabular}{ccccc}
    n & $\parallel A\parallel_2$ & $\epsilon_P$ &  $\epsilon_B$ &  $\epsilon_H$  \\\hline\hline
    32 & 7.6e+04& 1.2e-17& 4.9e-17&7.0e-22\\
    64 & 6.0e+05& 1.3e-17& 5.7e-17& 2.1e-22\\
    128 & 4.5e+06& 5.5e-18& 7.6e-17& 6.6e-24\\
    256 &3.6e+07& 6.6e-18& 1.3e-16& 1.5e-24\\
    512 &2.7e+08& 2.3e-18& 2.2e-16& 2.6e-25 \\ 
  \end{tabular}
  \caption{Backward errors for random  matrices of large norm  with $k=2$}
  \label{T22} 
  \end{center}
\end{table}

\begin{table} 
  \begin{center}
  \begin{tabular}{ccccc}
    n & $\parallel A\parallel_2$ & $\epsilon_P$ &  $\epsilon_B$ &  $\epsilon_H$  \\\hline\hline
    64 & 1.5e+02& 7.4e-18& 4.4e-17& 2.3e-18\\
    128 & 2.9e+02& 3.2e-18& 5.6e-17& 1.2e-18\\
    256 &5.5e+02& 2.5e-18& 9.6e-17& 4.1e-19 \\
    512 &1.1e+03& 1.8e-18& 1.6e-16& 5.0e-19 \\ 
  \end{tabular}
  \caption{Backward errors for random matrices with $k=4$}
  \label{T14} 
  \end{center}
\end{table}

\begin{table} 
  \begin{center}
  \begin{tabular}{ccccc}
    n & $\parallel A\parallel_2$ & $\epsilon_P$ &  $\epsilon_B$ &  $\epsilon_H$  \\\hline\hline
    64 & 6.2e+05& 6.6e-18& 5.2e-17& 5.3e-18\\
    128 & 4.9e+06& 4.5e-18& 6.8e-17& 1.8e-18\\
    256 &3.8e+07& 2.2e-18& 9.2e-17& 5.5e-19 \\
    512 &2.9e+08& 2.3e-18& 1.6e-16& 8.0e-19 \\ 
  \end{tabular}
  \caption{Backward errors for random matrices of large norm  with $k=4$}
  \label{T24} 
  \end{center}
\end{table}

The  results of Table \ref{T12},\ref{T22},\ref{T14} and \ref{T24}  show  that the
proposed algorithm is
numerically backward stable.

In order to confirm  the cost analysis of the algorithm
we have also performed
experiments taking   fixed the size of the matrix.  For
matrices of size $512$ with $k$ varying from 2 to 16 we obtain that
the measures of  elapsed time  $t_k$ satisfy
\[
\frac{t_4}{t_2}=2.34, \quad \frac{t_8}{t_4}=2.16, \quad \frac{t_{16}}{t_8}=2.08.
\]
This illustrates the linear growth of the cost with respect to $k$, the size of the
perturbation.

\section{Conclusions and Future Work}\label{sec:conclusions} 

In  this paper we have presented a novel algorithm for the reduction
in Hessenberg form of a  unitary  diagonal plus rank$-k$ matrix.
By exploiting the rank structure of the input matrix  this algorithm achieves
computational efficiency both with respect to the size of the matrix and the size of the perturbation
as well as numerical accuracy.   The algorithm
complemented  with the structured  QR iteration described in \cite{BDG} yields
a fast and accurate  eigensolver for unitary plus low rank matrices.

\bibliographystyle{siamplain}

\end{document}